\documentclass[12pt]{amsart}
\usepackage{amsmath,amssymb,amsthm}
\usepackage{graphicx}
\usepackage{hyperref}

\numberwithin{equation}{section}

\begin{document}

\newtheorem{theorem}{Theorem}[section]
\newtheorem{lemma}[theorem]{Lemma}
\newtheorem{corollary}[theorem]{Corollary}
\newtheorem{proposition}[theorem]{Proposition}
\newtheorem{definition}[theorem]{Definition}
\newtheorem{example}[theorem]{Example}
\newtheorem{conjecture}[theorem]{Conjecture}
\newtheorem{claim}[theorem]{Claim}
\newtheorem{xca}[theorem]{Exercise}

\newcommand{\ind}{\mbox{ind}}
\newcommand{\conn}{\mbox{connectivity}}
\newcommand{\Hom}{\mbox{Hom}}

\theoremstyle{remark}
\newtheorem{remark}[theorem]{Remark}
\newcommand{\be}{\begin{equation}}
\newcommand{\ee}{\end{equation}}
\newcommand{\lexmin}{\mbox{lexmin}}
\newcommand{\prob}{\mbox{\bf Pr}}
\newcommand{\var}{\mbox{Var}}
\newcommand{\cov}{\mbox{Cov}}
\newcommand{\diam}{\mbox{diam}}
\newcommand{\nerve}{\mathcal{N}}
\newcommand{\Star}{\mbox{st}}
\newcommand{\supp}{\mbox{supp}}
\newcommand{\vsupp}{\mbox{vsupp}}
\newcommand{\link}{\mbox{lk}}
\newcommand{\Z}{\mathbb{Z}}
\newcommand{\R}{\mathbb{R}}
\newcommand{\Q}{\mathbb{Q}}
\newcommand{\Homology}{\widetilde{H}}
\newcommand{\bound}{\partial}
\newcommand{\XG}{X(n,p)}

\title{Random geometric complexes}

\author{Matthew Kahle}
\address{Department of Mathematics, Stanford University}
\email{mkahle@math.stanford.edu}
\thanks{Supported in part by Stanford's NSF-RTG grant in geometry \& topology}

\date{\today}

\maketitle


\begin{abstract}

We study the expected topological properties of \v{C}ech and Vietoris-Rips complexes built on random points in $\R^d$. 
We find higher dimensional analogues of known results for connectivity and component counts for random geometric graphs. 
However, higher homology $H_k$ is not monotone when $k >0$.

In particular for every $k>0$ we exhibit two thresholds, one where homology passes from vanishing to nonvanishing, and another where 
it passes back to vanishing.  We give asymptotic formulas for the expectation of the Betti numbers in the sparser regimes, and bounds in the denser
regimes.  The main technical contribution of the article is the application of discrete Morse theory in geometric probability.

\end{abstract}

\section{Introduction}

The random geometric complexes studied here are simplicial complexes built on an i.i.d.\ random points in Euclidean space $\R^d$. We identify here the basic topological features of these complexes. In particular, we identify intervals of vanishing and non-vanishing for 
each homology group $H_k$, and give asymptotic formulas for the expected rank of homology when it is non-vanishing.

There are several motivations for studying this. The area of topological data analysis has been very active lately \cite
{Afra_computing, Edelsbrunner_survey}, and there is a need for a 
probabilistic null hypothesis to compare with topological statistics of point cloud data \cite{Carlsson}.

One approach to this problem was taken by Niyogi, Smale, and Weinberger \cite{Smale}, who studied the model where $n$ points are sampled uniformly and independently from a 
compact manifold $M$ embedded in $\R^d$, and estimates were given for how large $n$ must be in order to ``learn'' the topology of 
$M$ with high probability.  Their approach was to take balls of radius $r$ centered at the $n$ points and approximate the manifold by 
the \v{C}ech complex; provided that $r$ is chosen carefully, once there are enough balls to cover the manifold, one has a finite 
simplicial complex with the homotopy type of the manifold so in particular one can compute homology groups and so on.

The main technical innovation in \cite{Smale} is a geometric method for bounding above the number of random balls needed to cover the manifold, given some information about the curvature of the manifold's embedding. The  
assumption here is that one already knows how large $r$ must be, or that one at least has enough information about 
the geometry of the embedding of $M$ in order to determine $r$. (In a second article, they are able to 
recapture the topology of the manifold, even in the more difficult setting when Gaussian noise is added to every sampled point \cite{NSW2}.  Still, one needs some information about the embedding of the manifold.)

In this article we study both random Vietoris-Rips and \v{C}ech complexes for fairly general distributions on 
Euclidean space $\R^d$, and most importantly, allowing the radius of balls $r$ to vary from $0$ to $\infty$. We identify thresholds 
for non-vanishing and vanishing of homology groups $H_k$ and also derive asymptotic formulas and bounds on expectations of the 
Betti numbers $\beta_k$ in terms of $n$ and $r$.  It is well understood in computational topology that persistent homology is more 
robust than homology alone (see for example the stability results of Cohen-Steiner, Edelsbrunner, and Harer \cite{Edelsbrunner}), and 
one might not know anything about the underlying space, so in practice one computes persistent homology over a wide regime of 
radius \cite{Afra_computing}.

There is also a close connection to geometric probability, and in particular the theory of geometric random graphs.
Some of our results are higher-dimensional analogues of thresholds for connectivity and component counts in random geometric graphs
due to
Penrose \cite{Penrose}, and we must also use Penrose's results several times.  However, an important contrast is that the properties 
studied here are decidedly non-monotone. In particular, for each $k$ there is an interval of radius $r$ for which the homology group 
$H_k \neq 0$, and with the expected rank of homology $E[\beta_k]$ roughly unimodal in the radius $r$, but we also show that for 
large enough or small enough radius, $H_k=0$.

This paper can also be viewed in the context of several recent articles on the topology of random simplicial complexes \cite
{Linial_Meshulam, Meshulam_Wallach, BHK, Kahle_neighborhood, Kahle_clique, Pippenger}.  This article discusses a fairly 
general framework for random complexes, since one has the freedom to choose the underlying density function, hence an infinite-
dimensional parameter space.

The probabilistic method has given non-constructive existence proofs, as well as many interesting and extremal examples in 
combinatorics \cite{Alon}, geometric group theory \cite{gromov_random}, and discrete geometry \cite{Linial_Novik}.  Random 
spaces will likely provide objects of interest to topologists as well.

The problems discussed here were suggested, and the basic regimes described, in Persi Diaconis's MSRI 
talk in 2006 \cite{Persi_MSRI}. Some of the results in this article may have been discovered concurrently and independently by other 
researchers; it seems that Yuliy Barishnikov and Shmuel Weinberger have also thought about similar things \cite{Yuliy}. However, we 
believe that this article fills a gap in the literature and hope that it is useful as a reference.

\subsection{Definitions}

We require a few preliminary definitions and conventions.

\begin{definition} For a set of points $X \subseteq \R^d$, and positive distance $r$, define the {\it geometric graph} $G( X ; r)$ as 
the graph with vertices $V(G)=X$ and edges $E(G)=\{ \{x,y \} \mid d(x,y) \le  r \}$.
\end{definition}

\begin{definition} Let $f : \R^d \to \R$ be a probability density function, let $x_1, x_2, \ldots$ be a sequence of independent and 
identically distributed $d$-dimensional random variables with common density $f$, and let $X_n=\{x_1, x_2, \ldots, x_n\}$. The {\it 
geometric random graph} $G(X_n;r)$ is the geometric graph with vertices $X_n$, and edges between every pair of vertices $u,v$ 
with $d(u,v) \le r$.
\end{definition}

Throughout the article we make mild assumptions about $f$, in particular we assume that $f$ is a bounded Lebesgue-measurable function, 
and that $$\int_{\R^d} f(x) dx = 1$$ (i.e.\ that $f$ actually is a probability density function).

In the study of geometric random graphs \cite{Penrose} $r$ usually depends on $n$,  and one studies the asymptotic behavior of the graphs as $n \to \infty$.

\begin{definition} We say that $G(X_n; r_n)$ {\it asymptotically almost surely (a.a.s.)} has property $\mathcal{P}$ if $$\prob ( G
(X_n;r ) \in \mathcal{P} ) \to 1$$ as $n \to \infty$. \end{definition}

The main objects of study here are the \v{C}ech and Vietoris-Rips complexes on $X_n$, which are simplicial complexes built on the 
geometric random graph $G(X_n;r)$.  A historical comment:  the Vietoris-Rips complex was first introduced by Vietoris in 
order to extend simplicial homology to a homology theory for more general metric spaces \cite{Vietoris}.  Eliyahu Rips applied the same complex to the study of hyperbolic groups, and Gromov popularized the name Rips complex \cite{hyperbolic}.  The name ``Vietoris-Rips complex'' is apparently due to Hausmann \cite{Hausmann}.

Denote the closed ball of radius $r$ centered at a point $p$ by $B(p,r)=\{x \mid d(x,p) \le r \}$.

\begin{definition} The {\it random \v{C}ech complex} $C(X_n;r)$ is the simplicial complex with vertex set $X_n$, and $\sigma$ 
a face of $C(X_n;r)$ if $$\bigcap_{x_i \in  \sigma} B(x_i,r/2) \neq \emptyset.$$
\end{definition}
\begin{definition} The {\it random Vietoris-Rips complex} $R(X_n;r)$ is the simplicial complex with vertex set $X_n$, and $
\sigma$ a face if $$B(x_i,r/2) \cap B(x_j,r/2) \neq \emptyset$$ for every pair $x_i,x_j \in \sigma$.\\
\end{definition}
Equivalently, the random Vietoris-Rips complex is the clique complex of $G(X_n;r)$.

We are interested in the topological properties, in particular the vanishing and non-vanishing, and expected rank of homology groups, 
of the random \v{C}ech and Vietoris-Rips complexes, as $r$ varies. Qualitatively speaking, the two kinds of complexes behave very 
similarly. However there are important quantitative differences and one of the goals of this article is to point these out.

Throughout this article, we use Bachmann-Landau big-$O$, little-$O$, and related notations. In particular, for non-negative functions 
$g$ and $h$, we write the following.

\begin{itemize}
\item $g(n) = O(h(n))$ means that there exists $n_0$ and $k$ such that for $n > n_0$, we have that $g(n) \le k \cdot h(n)$. (i.e.\ $g$ is 
asymptotically bounded above by $h$, up to a constant factor.) 
\item $g(n) = \Omega(h(n))$ means that there exists $n_0$ and $k$ such that for $n > n_0$, we have that  $g(n) \ge k \cdot h(n)$. (i.e.
\ $g$ is asymptotically bounded below by $h$, up to a constant factor.) 
\item $g(n) = \Theta(h(n))$ means that $g(n) = O(h(n))$ and $g(n) = \Omega(h(n))$. (i.e.\ $g$ is asymptotically bounded above and 
below by $h$, up to constant factors.)
\item $g(n) = o(h(n))$ means that for every $\epsilon > 0$, there exists $n_0$ such that for $n > n_0$, we have that  $g(n) \le \epsilon 
\cdot h(n)$. (i.e.\ $g$ is dominated by $h$ asymptotically.)
\item $g(n) = \omega(h(n))$ means that for every $k >0$, there exists $n_0$ such that for $n > n_0$, we have that  $g(n) \ge k \cdot h
(n)$. (i.e.\ $g$ dominates $h$ asymptotically.)
\end{itemize}

When we discuss homology $H_k$ we mean either simplicial homology or singular homology, which are isomorphic.  Our results 
hold with coefficients taken over any field.

Finally, we use $\mu(S)$ to denote Lebesgue measure for any measurable set $S \subset \R^d$, and $\| x \|$ to denote the Euclidean 
norm of $x \in \R^d$.

\section{Summary of results}

It is known from the theory of random geometric graphs \cite{Penrose} that there are four main regimes of parameter (sometimes 
called {\it regimes}), with qualitatively different behavior in each. The same is true for the higher dimensional random complexes we 
build on these graphs. The following is a brief summary of our results.

In the subcritical and critical regimes, our results hold fairly generally, for any distribution on $\R^d$ with a bounded measurable 
density function.

In the subcritical regime, $r =o( n^{-1/d})$, the random geometric graph $G(X_n ; r)$ (and hence the simplicial complexes we are 
interested in) consists of many disconnected pieces. Here we exhibit a threshold for $H_k$, from vanishing to non-vanishing, and 
provide an asymptotic formula for the $k$th Betti number $E[\beta_k]$, for $k \ge 1$. 

In the critical regime, $r = \Theta (n^{-1/d})$, the components of the random geometric graph start to connect up and the giant 
component emerges. In other words, this is the regime wherein percolation occurs, and it is sometimes called the thermodynamic limit. 
Here we show that $E[\beta_k]=\Theta(n)$ and $\var[\beta_k]=\Theta(n)$ for every $k$.

The results in the subcritical and critical regimes hold fairly generally, for any distribution on $\R^d$ with a bounded measurable 
density function. In the supercritical and connected regimes, our results are for uniform distributions on smoothly bounded convex 
bodies in dimension $d$.

In the supercritical regime, $r = \omega( n^{-1/d})$. We put an upper bound on $E[\beta_k]$ to show that it grows sub-linearly, so the 
linear growth of the Betti numbers in the critical regime is maximal. Here our results are for the Vietoris-Rips complex, and the 
method is a Morse-theoretic argument.  The combination of geometric probability and discrete Morse theory used for these bounds is 
the main technical contribution of the article.

The connected regime, $r = \Omega(( \log{n} / n)^{1/d})$, is where $G(X_n;r)$ is known to become connected \cite{Penrose}. In 
this case we show that the \v{C}ech complex is contractible and the Vietoris-Rips complex is approximately contractible, in 
the sense that it is $k$-connected for any fixed $k$.  (This means that the homotopy groups $\pi_i$ vanish for $i \le k$, which implies
in turn that the homology groups $H_i$ vanish for $i \le k$ as well.)\\

Despite non-monotonicity, we are able to exhibit thresholds for vanishing of $H_k$.  For every $k \ge 1$, there is an interval in which $H_k \neq 0$ and outside of which $H_k = 0$, so every higher homology group passes through two thresholds.

The rest of the article is organized as follows. In Section \ref{sect:sub} we consider the subcritical regime of radius, in Section \ref
{sect:crit} the critical regime, in Section \ref{sect:sup} the supercritical regime, and in Section \ref{sect:con} the connected regime. In 
Sections \ref{sect:sup} and \ref{sect:con}  we assume that the underlying distribution is uniform on a smoothly bounded convex body 
mostly as a matter of convenience, but similar methods should apply in a more general setting. In Section \ref{sect:further} we discuss 
open problems and future directions.

\section{Subcritical} \label{sect:sub}

In this regime, we exhibit a vanishing to non-vanishing threshold for homology $H_k$, and in the non-vanishing regime compute the 
asymptotic expectation of the Betti numbers $\beta_k$, for $k \ge 1$. (The case $k=0$, the number of path components, is examined 
in careful detail by Penrose \cite{Penrose}, Ch. 13.) As a corollary, we also obtain information about the threshold where homology 
passes from vanishing to non-vanishing homology. We emphasize that the results in this section do not depend in any essential way on 
the distribution on $\R^d$, although we make the mild assumption that the underlying density function $f$ is bounded and 
measurable.

\subsection{Expectation}

\begin{theorem}\label{expect_Betti_Rips} [Expectation of Betti numbers, Vietoris-Rips complex] For $d \ge 2$, $k \ge 1$, $
\epsilon>0$, and $r _n = O(n^{-1/d - \epsilon})$, the expectation of the $k$th Betti number $E[\beta_k]$ of the random Vietoris-Rips 
complex $R(X_n;r)$ satisfies $$ \frac{E[\beta_k]}{n^{2k+2} r^{d(2k+1)}}  \to C_k,$$ as $n \to \infty$ where $C_k$ is a 
constant that depends only on $k$ and the underlying density function $f$.
\end{theorem}

(We note that this result holds for all $k$, even when $k \ge  d$.)

Using similar methods, we also prove the following about the random \v{C}ech complex.

\begin{theorem}\label{expect_Betti-Cech} [Expectation of Betti numbers, \v{C}ech complex] For $d \ge 2$, $1 \le k \le d-1$, $
\epsilon>0$, and $r = O(n^{-1/d - \epsilon})$, the expectation of the $k$th Betti number $E[\beta_k]$ of the random \v{C}ech 
complex $C(X_n;r)$ satisfies $$ \frac{E[\beta_k]}{n^{k+2} r^{d(k+1)}}  \to D_k,$$ as $n \to \infty$ where $D_k$ is a constant 
that depends only on $k$ and the underlying density function $f$. \end{theorem}

One feature that distinguishes the \v{C}ech complex from the Vietoris-Rips complex is that a {C}ech complex is always
homotopy equivalent to whatever it covers (this follows form the nerve theorem, i.e.\ Theorem 10.7 in \cite{Bjorner}).
So in particular $H_k=0$ when  $k \ge d$.

In both cases we will see that almost all of the homology is contributed
from a single source: whatever is the smallest possible vertex support for nontrivial homology.  For the Vietoris-Rips
complex this will be the boundary of the cross-polytope, and for the \v{C}ech complex the empty simplex.

\begin{definition}
The $(k+1)$-dimensional {\it cross-polytope} is
defined to be the convex hull of the $2k+2$ points $\{ \pm e_i
\}$, where $e_1, e_2, \ldots, e_{k+1}$ are the standard basis vectors of
$\R^{k+1}$. The boundary of this polytope is a $k$-dimensional simplicial
complex, denoted $O_{k}$.
\end{definition}

Simplicial complexes which arise as clique complexes of graphs are
sometimes called {\it flag complexes}.  A useful fact in
combinatorial topology is the following; for a proof see
\cite{Kahle_clique}.

\begin{lemma} \label{octa} 
If $\Delta$ is a flag complex, then any nontrivial element of $k$-dimensional
homology $H_k(\Delta)$ is supported on a subcomplex $S$ with at least
$2k+2$ vertices. Moreover, if $S$ has exactly $2k+2$ vertices, then
$S$ is isomorphic to $O_k$.
\end{lemma}

We also use results for expected subgraph counts in geometric random graphs.

Recall that a subgraph $H \le G$ is said to be an {\it induced} subgraph if for every pair of vertices $x,y \in V(H)$, we have $\{x,y\}$ 
is an edge of $H$ if and only if $\{ x,y \}$ is an edge of $G$. 

\begin{definition}A connected graph is {\it feasible} if it is geometrically realizable as an induced subgraph.
\end{definition}
For example the complete bipartite graph $K_{1,7}$ is not feasible, since it is not geometrically realizable as an induced subgraph of 
a geometric graph in $\R^2$, since there must be at least one edge between the seven degree-one vertices. 

Denote the number of induced subgraphs of $G(X_n;r)$ isomorphic to $H$ by $G_n(H)$, and the number of components isomorphic to $H$ by $J_n(H)$. Recall that $f$ is the 
underlying density function. For a feasible subgraph $H$ of order $k$, and $\mathcal{Y} \in (\R^d)^k$ define the indicator function 
$h_H(\mathcal{Y})$ on sets $\mathcal{Y}$ of $k$ elements in $\R^d$ by $h_H(\mathcal(Y))=1$ if the geometric graph $G(Y,1)$ is 
isomorphic to $H$, and $0$ otherwise. Let $$\mu_H = k!^{-1} \int_{\R^d} f(x)^k dx \int_{(\R^d)^{k-1}}h_H(\{ 0, x_1, \ldots, x_
{k-1} \} ) d(x_1,\ldots x_{k-1}).$$ Penrose proved the following \cite{Penrose}.

\begin{theorem}[Expectation of subgraph counts, Penrose] \label{expect_sub} Suppose that $\lim_{n \to \infty} (r) =0$, and $H$ 
is a connected feasible graph of order $k \ge 2$. Then $$\lim_{n \to \infty} r ^ {-d(k-1)}n^{-k} E(G_n(H)) = \lim_{n \to \infty} r 
^ {-d(k-1)}n^{-k} E(J_n(H)) = \mu_H. $$
\end{theorem}

Together with our topological and combinatorial tools, Theorem \ref{expect_sub} will be sufficient to prove Theorem \ref
{expect_Betti_Rips}.  To prove Theorem \ref{expect_Betti-Cech} we also require a hypergraph analogue of Theorem \ref
{expect_sub}, established by the author and Meckes in Section 3 of \cite{Meckes}, which we state when it is needed.\\
\medskip

\begin{proof}[Proof of Theorem \ref{expect_Betti_Rips}]

The intuition is that in the sparse regime, almost all of the homology is contributed by vertex-minimal spheres.

\begin{definition} 
For a simplicial complex $\Delta$, let $o_k(\Delta)$ (or $o_k$ if context is clear) denote the number of
induced subgraphs of $\Delta$ 
combinatorially isomorphic to the
$1$-skeleton of the cross-polytope $O_k$, and let $\widetilde{o}_k(
\Delta)$ denote the 
number of components of $\Delta$ combinatorially isomorphic to the
$1$-skeleton of the cross-polytope $O_k$.
\end{definition}

\begin{definition} 
Let $f_k^{= i}(\Delta)$ denote the number of $k$-dimensional faces on
connected components with exactly $i$ vertices.
Similarly,
let $f_k^{\ge i}(\Delta)$ denote the number of $k$-dimensional faces
on connected components containing at least $i$ vertices.
\end{definition}

%
%
%
%
%
 
A dimension bound paired with Lemma \ref{octa} yields 
\begin{equation}\label{octo-morse}
\widetilde{o}_k \le \beta_k \le \widetilde{o}_k + f_k^{\ge 2k+3}.
\end{equation}

One could work with $f_k^{\ge 2k+3}$ directly, but it turns out to be sufficient 
to overestimate $f_k^{\ge 2k+3}$ as follows.  For each 
$k$-dimensional face in a component
with at least $2k+3$ vertices, extend to a connected
subgraph with exactly $2k+3$ vertices and ${k+1 \choose 2} + k+2$
edges.

For example, let $k=2$; then  
 \begin{equation} \label{eqn_squeeze}
 \widetilde{o}_2 \le \beta_2 \le \widetilde{o}_2 + f_2^{\ge 7}.
 \end{equation}
   Up to isomorphism, the seventeen graphs that arise when
extending a $2$-dimensional face (i.e.\ a $3$-clique) to a minimal
connected graph on $7$ vertices are exhibited in Figure
\ref{fig:betti2}.

In particular, $f_2^{\ge 7} \le \sum_{i=1}^{17} s_i,$ where
$s_i$ counts the number of subgraphs isomorphic to graph $i$ for some
indexing of the seventeen graphs in Figure \ref{fig:betti2}.

\begin{figure}
\begin{centering}
\includegraphics[width=4.5in]{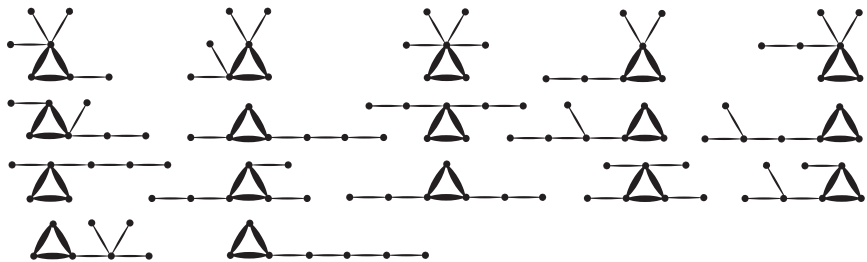}
\end{centering}
\caption{The case $k=2$: the seventeen isomorphism types of subgraphs
  which arise when extending a $3$-clique to a connected graph on $7$
  vertices with $7$ edges.  Each
  subgraph isomorphic to one of these can contribute at most $1$
  to the sum bounding the error term $f_2^{\ge 7}$.}
\label{fig:betti2}
\end{figure}

Moreover, as noted in \cite{Penrose}, 
the number of occurences of a given subgraph $\Gamma$ on $v$ vertices is 
a positive linear combination of the induced subgraph counts for those
graphs on $v$ vertices which have $\Gamma$ as a subgraph.

For an example of this, let $G_{H}$ denote the number of induced subgraphs of $G$ isomorphic to $H$, and let $\widetilde{G}_{H}$ denote
the number of subgraphs (not necessarily induced) of $G$ isomorphic to $H$.
If $P_3$ is the path on $3$ vertices and $K_3$ is the complete graph on $3$ vertices, then
$$\widetilde{G}_{P_3} = 3 G_{K_3} + G_{P_{3}}.$$ 

So for each $i$ we can write $s_i$ as a positive linear combination of induced subgraph counts, and every type of induced subgraphs has exactly $7$ vertices.  

We take expectation of both sides of Equation \ref{eqn_squeeze}, applying linearity of expectation, to obtain
\begin{align*}
E[\widetilde{o}_2]  \le E[\beta_2] & \le E[\widetilde{o_2}]+E[f_2^{\ge 7}] \\
& \le E[\widetilde{o_2}]+E[ \sum_{i=1}^{17} s_i]\\
& \le E[\widetilde{o_2}]+ \sum_{i=1}^{17} E[s_i].
\end{align*}

For each $i$, $E[s_i] = O(n^7 r^{6d})$, by Theorem \ref{expect_sub}.  On the other hand, $E[\widetilde{o_2}] = \Theta (n^6 r^{5d})$, also by Theorem \ref{expect_sub}.
Since we are assuming that $n r^d \to 0$ as $n \to \infty$, we have shown that $E[f_2^{\ge 7}]=o(E[\widetilde{o}_2] )$.  We conclude that $E[\beta_2] / E[\widetilde{o}_2]  \to 1$ as $n \to \infty$.  This gives $E[\beta_2]=\Theta (n^6 r^{5d})$, as desired.

The proof for $k \ge 2$ is the same.  In general the number of graphs on $2k+3$ vertices that can 
arise from the algorithm above is a constant that only depends on $k$, so denote this constant by $c_k$.

So in general we will have 
\begin{align*}
E[\widetilde{o}_k]  \le E[\beta_k] & \le E[\widetilde{o_k}]+E[f_k^{\ge 2k+3}] \\
& \le E[\widetilde{o_k}]+E[ \sum_{i=1}^{c_k} s_i]\\
& \le E[\widetilde{o_2}]+ \sum_{i=1}^{c_k} E[s_i].
\end{align*}

For each $i =1, 2, \ldots, c_k$ we have $$E[s_i] = O(n^{2k+3} r^{(2k+2)d}),$$ and on the other hand $$E[\widetilde{o}_k] = \Theta(n^{2k+2}r^{(2k+1)d}).$$
Since $nr^d \to 0$, we conclude that $E[\beta_k] / E[\widetilde{o}_k]  \to 1$, and $$E[\beta_k]= \Theta (n^{2k+2} r^{(2k+1)d}).$$

The case $k=1$ is slightly different. There are several ways of extending a $2$-clique (i.e.\ an
edge) to a connected graph on $5$ vertices and $4$ edges.  In this
case the graph must be a tree, and there
are three isomorphism types of trees on five vertices, shown in
Figure \ref{fig:betti1}.  But in this case counting these subgraphs 
will result in an underestimate
for $f_1^{\ge 5}$.  However, each tree has only four edges, and so one can obtain the bound
 $$f_1^{\ge 5} \le 4 ( t_1 + t_2 + t_3),$$ where
$t_1, t_2, t_3$ count the number of subgraphs isomorphic to the three
trees in Figure \ref{fig:betti1}.  The argument is then the same as in the 
case $k\ge 2$.

\begin{figure}
\begin{centering}
\includegraphics[width=3.5in]{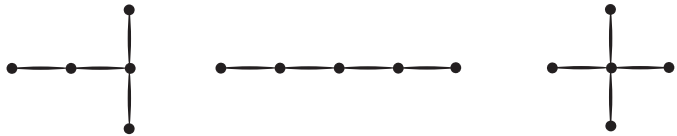}
\end{centering}
\caption{The case $k=1$: the three isomorphism types of trees on five vertices.  Each
  subgraph isomorphic to one of these can contribute at most $4$
  to the sum bounding the error term $f_1^{\ge 5}$.}
\label{fig:betti1}
\end{figure}

This completes the proof, modulo one small concern:  we must make sure that the octahedral $1$-skeletons are
geometrically feasible.  It is perhaps surprising that this is the case, even when $d=2$.  But the regular $2k$-gons
provide examples of geometic realizations of the $1$-skeleton of $O_k$ for every $k$, as in Figure \ref{fig:even}.
(This fact was previously noted by Chambers, de Silva, Erickson, and Ghrist in \cite{Rips_plane}.)

\begin{figure}
\begin{centering}
\includegraphics[width=4in]{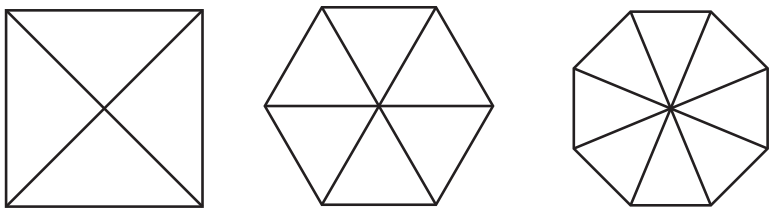}
\end{centering}
\caption{The regular $2k$-gons prove that the $1$-skeleton of the cross-poytope $O_k$ is
geometrically feasible in
the plane for every $k$. If $r$ is slightly shorter than the length of the main diagonal, components
combinatorially isomorphic to this contribute to $\beta_k$
in the Vietoris-Rips complex.}
\label{fig:even}
\end{figure}

\end{proof}

\begin{proof}[Proof of Theorem \ref{expect_Betti-Cech}] The argument for the \v{C}ech complex proceeds along the same lines, 
mutatis mutandis, but with one important difference.  Again the dominating contribution to $\beta_k$ will come from vertex-minimal 
$k$-dimensional spheres, but for a \v{C}ech complex the smallest possible vertex support that a simplicial complex with nontrivial 
$H_k$ can have is $k+2$ vertices, coming from the boundary of a $(k+1)$-dimensional simplex.

Let $\widetilde{S_k}$ denote the number of connected components isomorphic to 
the boundary of a $(k+1)$-dimensional simplex.  By the same argument as before we have
$$E[\widetilde{S_k}]  \le E[\beta_k]  \le E[\widetilde{S_k}]+E[f_k^{\ge k+3}] .$$

Deciding whether some set of $k+2$ vertices span the boundary of a $(k+1)$-dimensional simplex depends on 
higher intersections, so in particular when $k > 2$ the faces of the \v{C}ech complex are not determined by the underlying geometric graph.  It is proved in Section 3 of \cite{Meckes} that as long as $r = o(n^{-1/d})$ then $E[\widetilde{S_k}] = \Theta( n^{k+2}r^{(k+1)d})$.  On the other hand we have 
$E[f_k^{\ge k+3}] = O ( n^{k+3}r^{(k+2)d})$.  As before, since $r = o(n^{-1/d})$ this is enough to give that $$ \lim_{n \to \infty} E[\beta_k]  / E[\widetilde{S_k}]   = 1,$$ and then $E[\beta_k] = \Theta ( n^{k+2}r^{(k+1)d})$ as desired.

\end{proof}

\subsection{Vanishing / non-vanishing threshold}

To state the following theorems we assume that $d \ge 2$ and $k \ge 1$ are fixed and that $r$ is still in the sparse regime, i.e.\ that 
$r=o(n^{-1/d})$.

\begin{theorem} [Threshold for non-vanishing of $H_k$ in the random Vietoris-Rips complex] \label{lower-thresh-VR}

\

\begin{enumerate}

\item If $$r = o\left( n^{-\frac{2k+2}{d(2k+1) }}\right) ,$$ then a.a.s.\ $H_k(VR(n;r))=0$, and 

\item if  $$r = \omega \left( n^{-\frac{2k+2}{d(2k+1) }}\right) ,$$ then a.a.s.\ $ H_k(VR(n;r)) \neq 0$.

\end{enumerate} \end{theorem}


\begin{proof}

The first statement follows directly from Lemma \ref{octa} and Theorem \ref {expect_sub}; i.e.\ if $r$ is too small then the 
connected components are simply too small to support nontrivial homology.

For the second statement, we have from Theorem \ref{expect_Betti_Rips} that given this hypothesis on $r$ we have that $E
[\beta_k] \to \infty$. This by itself is not enough to establish that $\beta_k \neq 0$ a.a.s.\  However it is established in Section 4 of \cite{Meckes} 
that $\var[ \beta_k]$ is of the same order of magnitude as $E[\beta_k]$, so this follows from Chebyshev's inequality, as in \cite
{Alon}, Chapter 4.

\end{proof}

The corresponding result for \v{C}ech complexes is the following.

\begin{theorem} [Threshold for non-vanishing of $H_k$ in the random \v{C}ech complex] \label{lower-thresh-C}

\

\begin{enumerate}

\item If $$r = o\left( n^{-\frac{k+2}{d(k+1) }}\right) ,$$ then a.a.s.\ $H_k(VR(n;r))=0$, and 

\item if  $$r = \omega \left( n^{-\frac{k+2}{d(k+1) }}\right) ,$$ then a.a.s.\ $ H_k(VR(n;r)) \neq 0$.

\end{enumerate} \end{theorem}

\begin{proof}
The proof is identical.  The needed result for bounding the variance of  $\var[ \beta_k]$  is established in Section 3 of \cite{Meckes}.
\end{proof}

\section{Critical} \label{sect:crit}

The situation in the critical regime (or thermodynamic limit) is more delicate to analyze. We are still able to compute the right order of magnitude for $E[\beta_k]$: it grows 
linearly for every $k$.

\begin{theorem} For either the random Vietoris-Rips and \v{C}ech complexes on a probability distribution on $\R^d$ with bounded measurable density function, if $r=\Theta(n^{-1/d})$ and $k \ge 1$ is fixed, then $E[\beta_k] = \Theta(n)$.
\end{theorem}

\begin{proof} The proof is the same as in the previous section.  For example, for the Vietoris-Rips complex we still have 
$$E[\widetilde{o}_k]  \le E[\beta_k]  \le E[\widetilde{o_k}]+E[f_k^{\ge 2k+3}] .$$  Penrose's results for component counts extend in to the thermodynamic limit, so in particular $E[\widetilde{o}_k] = \Theta(n)$ and  $E[f_k^{\ge 2k+3}]=O(n)$.  The desired result follows.
\end{proof}

The thermodynamic limit is of particular interest since this is the regime where percolation occurs for the random geometric graph 
\cite{Penrose}.  Bollob\'as recently exhibited an analogue of percolation on the $k$-cliques of the Erd\H{o}s-R\'enyi random graph 
\cite{Bollo_clique}. It would be interesting to know if analogues of his result occurs in the random geometric setting.

For example, define a graph with vertices for $k$-dimensional faces, with edges between a pair whenever they are both contained in 
the same $(k+1)$-dimensional face. Does there exist a constant $C_k >0$ such that whenever $$\lim_{n \to \infty} nr^d >  C_k$$ there is a.a.s.\ a unique $k$-dimensional ``giant component'' (suitably defined), and whenever $$\lim_{n \to \infty} nr^d <  C_k,$$ all the 
components are a.a.s.\ ``small''?

\section{Supercritical} \label{sect:sup}

For this section and the next we assume that the underlying distribution is uniform on a smoothly bounded convex body.  (Recall that a smoothly bounded
convex body is a compact, convex set, with nonempty interior.) This assumption is not 
only a matter of convenience -- it would seem that some assumption on density must be made to make topological statements 
in the denser regimes.

For example, the geometric random graph becomes connected once $r = \Omega((\log{n}/n)^{1/d} )$ for a 
uniform distribution on a convex body, but for a standard multivariate normal distribution $r$ must be much larger, $r = \Omega
( (\log{\log{n}} / \log{n} )^{1/2} )$, before the geometric random graph becomes connected \cite{Penrose}.

The supercritical regime is where $r = \omega (n^{-1/d})$. In this section 
we give an upper bound on the expectation of the Betti numbers for the random Vietoris-Rips complex in this regime.  This upper bound is sub-linear
so this shows that the Betti numbers are growing the fastest in the thermodynamic limit.

The main tool is discrete Morse theory -- see the Appendix for the basic terminology and the main theorem.
A much more complete (and very readable) introduction to discrete Morse theory can be found in \cite{Forman}.

\begin{theorem}\label{thm-Ripssup} Let $R(X_n;r)$ be a random Vietoris-Rips complex on $n$ points taken i.i.d. uniformly from a smoothly bounded 
convex body $K$ in $\R^d$. Suppose $r = \omega (n^{-1/d})$, and write $W=nr^d$.  Then $$E[\beta_k] 
= O(W^k e^{-c W} n)$$ for some constant $c > 0$, and in particular $E[\beta_k] = o(n)$.
\end{theorem}

Here $c$ depends on the convex body $K$ but not on $k$. In fact it is apparent from the proof that $c$ depends only on the volume of $K$ and not on its shape.

Recall that $\mu(S)$ denotes the Lebesgue measure of $S \subset \R^d$, and $\| x \|$ denotes the Euclidean norm of $x \in \R^d$. 
We require a geometric lemma in order to prove the main theorem.

\begin{lemma}[Main geometric lemma] \label{Mgl} There exists a constant $\epsilon_d > 0$ such that the following holds. Let $l \ge 1$ and
$\{ y_0, \ldots, y_l \} \subset \R^d$ be an $(l+1)$-tuple of points such that
$$ \| y_0 \|  \le \| y_1 \| \le \ldots \le \|y_l \|,$$
and  $\| y_1\| \ge 1/2 $.  If $\| y_0 -  y_1 \| > 1$ and $\| y_i -  y_j \| \le 1$ for every other $0 \le i < j \le l$, then the intersection
 $$I = \bigcap_{i=1}^l B(y_i,1) \cap B(0,\| y_1 || )$$ satisfies $\mu(I) \ge \epsilon_d$.
 \end{lemma}

As the notation suggests, $\epsilon_d$ depends on $d$ but holds simultaneously for all $l$.

\begin{proof} [Proof of Lemma \ref{Mgl}] Let $y_m = (y_0 + y_1) /2$ denote the midpoint of line segment $y_0y_1$. By assumption 
that $\| y_0 - y_1 \| > 1$, we have $\| y_m -y_0 \| = \| y_m - y_1 \| > 1/2$. We now wish to check that $y_m$ is still not too far away 
from any $y_i$ with $2 \le i \le l$.

Let $\theta$ be the positive angle between $y_0 -y_2$ and $y_1 - y_2$. Since $\| y_0 - y_2 \| \le 1$, $\| y_1 - y_2 \| \le 1$, and $\| y_1 
- y_2 \| > 1$, the law of cosines gives that 
\begin{align*}
(y_0 - y_2) \cdot (y_1 - y_2)  &  = \| y_0 - y_2 \| \| y_1 - y_2 \| \cos{\theta} \\
& = \frac{1}{2} (\| y_0 -y_2 \|^2 + \|y_1 - y_2\|^2 -\|y_0 - y_2\|^2)  \\
&< \frac{1}{2}
\end{align*}
Then
\begin{align*}
\| y_m - y_2 \| ^2 & =  (y_m -y_2) \cdot (y_m - y_2)\\
& = ((y_0 +y_1) / 2 -y_2) \cdot ((y_0 +y_1) / 2 -y_2)\\
& = ((y_0 - y_2) / 2 + (y_1 - y_2) / 2 ) \cdot ((y_0 - y_2) / 2 + (y_1 - y_2) / 2 ) \\
& = (1/4) ( \| y_0 - y_2 \| ^2 + \| y_1 - y_2 \| ^2  + 2 (y_0 - y_2) \cdot (y_1 - y_2) )\\
& < (1/4) (1 + 1 + 2(1/2))\\
& = 3/4,\\
\end{align*}
so  $$ \| y_m - y_2 \| < \sqrt{3} / 2.$$ The same argument works as written with $y_2$ replaced by $y_i$ with $3 \le i \le l$. Now set $
\rho = 1 - \sqrt{3} / 2 $. By the triangle inequality $B( y_m , \rho) \subset B( y_i , 1)$ for $1 \le i \le l$. 
So we have that $$B(y_m , \rho) \cap B(0 , \| y_1\|) \subset \bigcap_{i=1}^l B(y_i ,1) \cap B(0, \| y_1 \| ).$$
By the triangle inequality we have that $\| y_m \| \le \| y_1 \|$, and it follows that
$$\mu \left( B(y_m , \rho) \cap B(0, \| y_1\|) \right) \ge \mu \left( B(y_1, \rho) \cap B (0 , \| y_1 \|) \right).$$
Since $\| y_1 \| \ge 1/2$, the quantity $\mu \left(B(y_1 , \rho) \cap B (0 , \| y_1 \|) \right)$ is bounded away from zero, and in fact it attains its 
minimum when $\| y_1 \| = 1/2$. Set $\epsilon_d$ equal to this minimum value of $\mu (B(y_1 , \rho) \cap B (0 , \| y_1 \|) )$, and the 
statement of the lemma follows.\\
\end{proof}

Scaling everything in $\R^d$ by a linear factor of $r$ we rewrite the lemma in the form in which we will use it.

\begin{lemma} \label{scaled} [Scaled geometric lemma] There exists a constant $\epsilon_d > 0$ such that the following holds for every $r >0$. Let 
$l \ge 1$ and $\{ y_0, \ldots, y_k \} \subset \R^d$ be an $(l+1)$-tuple of points, , such that $$ \| y_0 \|  \le \| y_1 \| \le \ldots \le \|y_l \|$$ and 
$(1/2) r \le \| y_1 \|$. If $\| y_0 -  y_1 \| > r$ and $\| y_i -  y_j \| \le r$ for every other $0 \le i < j \le l$, then the intersection $$I = \bigcap_
{i=1}^l B(y_i,r) \cap B(0,\| y_1 \| )$$ satisfies $\mu(I) \ge \epsilon_d r^d$.
 \end{lemma}

We are ready to prove the main result of the section.

\begin{proof} [Proof of Theorem \ref{thm-Ripssup}] By translation and rescaling if necessary, assume without loss of generality that $B(0,1) \subset K$. Since with probability $1$ no two points are the same distance to the origin, index the points $X_n = \{x_1, \ldots, x_n\}$  by distance to $0$, i.e.
$$ \| x_1 \| < \| x_2 \| < \dots < \| x_n \|.$$ 
Now we  define a discrete vector field $V$ on $R(X_n;r)$ in the sense of discrete Morse theory, as discussed in the Appendix.


Whenever possible pair face $S = \{ x_{i_1}, x_{i_2}, \dots , x_{i_j} \}$  with face $\{ x_{i_0}\} \cup S$
 with $i_0 < i_1$ and $i_0$ as small as possible.  This can be done in any particular order or simultaneously, and
still each face gets paired at most once, as follows.  A face $S$ can
not get paired with two different higher dimensional faces $\{x_{a} \} \cup S$ and $\{ x_{b} \} \cup S$,  since $S$ will prefer the vertex with smaller index $\min \{a,b \}$.  On the other hand, it is also not possible for $S$ to get paired with both a lower dimensional face and a higher dimensional face:
Suppose $S$ gets paired with $\{ x_a \} \cup S$.   Then $\| x_a \| < \| s \|$ for every $s \in S$, and no codimension $1$ face $F \prec S$ could also get paired with $S$, since $F$ would prefer to get paired with  $\{ x_a \} \cup F$.

Hence each face is in at most one pair and $V$  is a well defined discrete vector field. Moreover, the indices are decreasing along any $V$-path, so there are no closed $V$-paths.  Therefore $V$ is a discrete gradient vector field.


Let us bound the probability $p_k$ that a set of $k+1$ vertices span a $k$-dimensional face in the Vietoris-Rips complex.  Given the first vertex $v$, 
the other vertices would all have to fall in $B(v,r)$, so $p_k=O(r^{dk})$.  Recall that we defined $W = nr^d$ and we rewrite this bound as
$$p_k = O \left( (W/n)^k \right).$$

Given that a set of $k+1$ vertices $\{ x_{i_1}, x_{i_2}, \dots, x_{i_{k+1}} \}$ span a $k$-dimensional face $F$, how could $F$ be critical (or unpaired) with respect to $V$?  It must be that there is no common neighbor $x_a$ of these vertices with $a < i_1$ or else $F$ would be paired up by adding the smallest index such point.  On the other hand $F$ would be paired with $\{ x_{i_2}, \dots, x_{i_{k+1}} \}$, unless $x_{i_2}, \dots, x_{i_{k+1}}$ had a common neighbor with smaller index than $x_{i_1}$.  So assuming that $F$ is unpaired call this common neighbor $x_{i_0}$.

We have satisfied the hypothesis of Lemma \ref{scaled} with $l = k+1$ and $y_m = x_{i_m}$. (If $ \| y_1 \|  < (1/2) r$
then either $\| y_0 - y_1 \| < r$ or $\| y_o \| > \| y_1 \|$, a contradiction to our assumptions.)  So let
$$I = \bigcap_{j=1}^{k+1} B(x_{i_k},r) \cap B(0,\| x_{i_1} \| ),$$ and we know from the lemma 
that $\mu(I) \ge \epsilon_d r^d$ with $\epsilon_d > 0$
constant.

If any vertices fall in region $I$ then $F$ would be paired; indeed if $x_a \in I$ then $x_a$ would be a common neighbor of all the vertices in $F$, with $a < i_1$.

The probability that a uniform random point in $K$ falls in region $I$ is $\mu(I) / \mu (K) \ge \epsilon_d r^d / \mu(K)$, where 
$\mu(K)$ is the volume of the ambient convex body.
By independence of the random points, we have that the probability $p_c$ that $F$ is critical (given that it is a face) is at most
$$p_c \le \left( 1 - \frac{\epsilon_d}{ \mu(K)}r^d \right)^{n-k-2}.$$

Now
\begin{align*}
\left( 1 -  \frac{\epsilon_d}{ \mu(K)} r^d \right)^{n-k-2} & \le \exp(- \frac{\epsilon_d}{ \mu(K)}r^d(n-k-2) ) \\
& = O( \exp(-cW)), \\
\end{align*}
where $c$ is any constant such that $$0 < c <\frac{\epsilon_d}{ \mu(K)}.$$

Let $C_k$ denote the number of critical $k$-dimensional faces, and we have that
\begin{align*}
E[ C_k ] & \le { n \choose k+1}p_f p_c\\
& \le {n \choose k+1} \left( \frac{W}{n} \right)^{k} e^{-c W}\\
&= O(W^k e^{-cW}n).
\end{align*}

Since $\beta_k \le C_k$ in every case we have $E[\beta_k] \le E[C_k]$, and then
$$E[\beta_k] = O(W^k e^{-cW}n),$$
as desired.

\end{proof}

%

\section{Connected} \label{sect:con}

As in the previous section, we assume that the underlying distribution is uniform on a smoothly bounded convex body $K$,
but we now require $r$ to be slightly larger,
$r = \Omega((\log{n}/n)^{1/d} )$. In this regime, the geometric random graph is known to be connected \cite{Penrose}, and we 
show here that the \v{C}ech complex is contractible, and the Vietoris-Rips complex ``approximately contractible'' (in the sense of $k
$-connected for any fixed $k$).

\begin{theorem}[Threshold for contractibility, random \v{C}ech complex] \label{contractible} For a uniform distribution on a smoothly bounded convex body $K$ in $\R^d$, there exists a constant $c$, depending on $K$, such that if 
$r \ge c (\log{n}/n)^{1/d}$ then the random \v{C}ech complex $C(X_n;r)$ is a.a.s.\ contractible.
\end{theorem}

This is best possible up to the constant in front, since there also exists a constant $c'$ such that if 
$r \le c' (\log{n}/n)^{1/d}$, then the random \v{C}ech complex is a.a.s.\  disconnected \cite{Penrose}.

\begin{definition} Let $\mathcal{A}= \{ A_1, A_2, \dots, A_k \}$  be a cover of a topological space $T$.  Then the
{\it nerve} of the cover $\mathcal{A}$, 
is the (abstract) simplicial complex $\nerve(\mathcal{A})$ on vertex set $[k] = \{ 1, 2, \dots, k \}$ with $\sigma \subset [k]$ a face
whenever $\bigcap_{i \in \sigma} A_i \neq \emptyset.$
\end{definition}

The proof depends on the following result (Theorem 10.7 in \cite{Bjorner}). 

\begin{theorem} [Nerve Theorem] \label{nerve} If $T$ is a triangulable topological space, and $\mathcal{A}= (A_i )_{i \in [k]}$ is a finite cover of $T$ by closed sets, such 
that every nonempty section $A_{i_1} \cap A_{i_2} \cap \dots \cap A_{i_t}$ is contractible, then $T$ and the nerve $\nerve(T)$ are homotopy
equivalent.
\end{theorem}

\begin{proof} [Proof of Theorem \ref{contractible}]
Once $r$ is sufficiently large the balls $\{ B(x_i,r/2) \}$ cover the smoothly bounded convex body $K$, and then
Theorem \ref{nerve}  gives that it is contractible.  So to prove the claim it suffices to show that there exists a constant $ c>0$ such that whenever
$r \ge c (\log{n}/n)^{1/d}$, the balls of radius $r/2$ a.a.s.\ cover $K$.  There is no harm in assuming that $r \to 0$ as $n \to \infty$ since the statement is trivial otherwise.

Let $\Z^d$ denote the $d$-dimensional cubical lattice, and $\lambda \Z^d$ the same lattice linearly scaled in every direction
by a factor $\lambda >0$.  With the end in mind we set $\lambda = r /( 4 \sqrt{d})$.  (Note that since $r=r(n)$, $\lambda$ is also a function of $n$.) Since $K$ is bounded, only a finite number $N$ of the boxes
of side length $\lambda$ intersect it.  More precisely, it is easy to see that $$N = \mu(K) / \lambda^d+O(1/\lambda^{d-1}).$$

As $n \to \infty$ and $\lambda \to 0$ almost all of these $N$ boxes are contained in $K$, but some are on the boundary.  Denote by $S_K$ the set of boxes completely contained in $K$.  Suppose every box in $S_K$ contains at least one point in $X_n$.  Then the balls of radius $r/2$ cover $K$, as follows.

First of all, each box has diameter $\lambda \sqrt{d} = r /4$. So a ball of radius $r/2$ with a point in one of these boxes not only covers the box itself, but all the boxes adjacent to it.  Since every boundary box is adjacent to at least one box in $S_K$, this is sufficient.

For a box $B \in S_K$, let $p_o$ denote the probability that box $B \cap X_n = \emptyset$.  By uniformity of distribution this is the same for every $B$, and by independence of the points we have that 
\begin{align*}
p_o & = (1 - \lambda^d / \mu(K))^n \\
& \le \exp( -\lambda^d n / \mu(K) ) \\
& = \exp ( - (r / 4\sqrt{d}) ^ d n/ \mu(K) )\\
& = \exp ( - C r^d n),
\end{align*}
where $$C=\frac{1}{4^d d^{d/2} \mu(K)}$$
is constant.

Setting $r = c_k (\log{n}/n)^{1/d}$ we have that
\begin{align*}
p_o &\le \exp(-C c_k^d \log{n})\\
& = n^{-C c_k^d}.\\
\end{align*}
There are at most $N$ boxes in $S_K$ and
\begin{align*}
N &= \mu(K) / \lambda^d+O(1/\lambda^{d-1})\\
 &= (1+o(1)) / C r^d,
\end{align*}
so applying a union bound, the probability $p_f$ that at least one box in $S_K$ fails to
contain any points from $X_n$ is bounded by
\begin{align*}
p_f  & \le N p_o \\
 & \le \frac{1+o(1)}{ C r^d}  n^{-C c_k^d}\\
& = \frac{1+o(1)}{ Cc_k^d \log{n}}  n^{1-C c_k^d}.
\end{align*}
So choosing $c_k > (1 / C)^{1/d}$ is sufficient to ensure that $K$ is a.a.s.\ covered by the $n$ random balls of radius $r/2$, and the desired result follows.
\end{proof}

The situation for the Vietoris-Rips complex is a bit more subtle since the nerve theorem is not available to us.  Nevertheless, we 
use Morse theory to show 
in the connected regime that the Vietoris-Rips complex becomes ``approximately contractible,'' in the sense of highly connected.

\begin{definition} A topological space $T$ is $k$-connected if every map from an $i$-dimensional sphere $S^i \to T$ is 
homotopically trivial for $0 \le i \le k$.
\end{definition}
For example, $0$-connected means path-connected, and $1$-connected means path-connected and simply-connected. The Hurewicz Theorem and universal coefficients for homology gives that if $T$ is $k$-connected, then $\widetilde{H}_i(T) =0$ for $i \le k$, with coefficients in $\Z$ or any field \cite{Hatcher}.

\begin{theorem}  [$k$-connectivity of the random Vietoris-Rips complex]  \label{Rips_conv_conn}For a smoothly bounded convex body $K$ in $\R^d$, endowed with a uniform 
distribution, and fixed $k \ge 0$, if $r \ge c_k (\log{n}/n)^{1/d}$ then the random Vietoris-Rips complex $R(X_n;r)$ is a.a.s.\ $k
$-connected. (Here $c_k >0$ is a constant depending only on the volume $\mu(K)$ and $k$.)
\end{theorem}

\begin{proof} [Proof of Theorem \ref{Rips_conv_conn}] 

The proof is identical to the proof of Theorem \ref{thm-Ripssup}, but now $r$ is bigger and we obtain a stronger result.  
We place a discrete gradient vector field on $R(X_n;r)$ in the same way described before, and repeat the same argument.
If $C_k$ denotes the number of critical $k$-dimensional faces, $c$ is the constant in the statement of Theorem \ref{thm-Ripssup},
and $W=nr^d$, then we have
\begin{align*}
E[ C_k ] &= O\left( W^k e^{-cW}n \right)\\
& = O\left( (nr^d)^k e^{-cnr^d} n \right)\\
& = O\left( (c_k^d \log{n} )^k  n^{1-c c_k^d} \right),
\end{align*}
since $n r^d   = c_k^d \log{n}$.  So 
we have that $E[ C_k] \to 0$ provided that $c_k > (1/c)^{1/d}$.

The same argument holds simultaneously for all smaller values of $k \ge 1$ as well, so a.a.s.\ the only critical cell of dimension $\le k$ is the vertex closest to the origin. By Theorem \ref{morse} in the Appendix, $R(X_n;r)$ is a.a.s.\ homotopy equivalent to a CW-complex with one $0$-cell and no other cells of dimension $\le k$.  This implies that $R(X_n;r)$ is $k$-connected by cellular approximation \cite{Hatcher}.
\end{proof}

At the moment we do not know if there is a sufficiently large constant $t >0 $ such that whenever $ r \ge t (\log{n} /n) ^{1/d}$, the random Vietoris-Rips complex $R(X_n; r)$ is a.a.s.\ contractible.  In fact it is not even clear that making $r = \omega \left(  (\log{n} /n) ^{1/d} \right)$ is sufficient for this; this ensures that $R(X_n ;r)$ is a.a.s.\ $k$-connected for every fixed $k$, but our results here do not rule out the possibility that there is nontrivial homology in dimension $k$ where $k \to \infty$ as $n \to \infty$. 

\subsection{Non-vanishing to vanishing threshold}

Given a lemma about geometric random graphs which we state without proof, we have a second threshold where $k$th homology 
passes back from non-vanishing.

First the statement of the lemma. (We are still assuming that the underlying distribution is uniform on a smoothly bounded convex 
body.)

\begin{lemma}  Suppose $H$ is a feasible subgraph, that $r =\Omega (n^{-1/d})$, and that $r = o \left(( \log n /n )^{1/d}\right)$.  
Then the geometric random graph $X(n;r)$ a.a.s.\ has at least one connected component isomorphic to $H$.
\end{lemma}

This lemma should follow from the techniques in Chapter 3 of \cite{Penrose}.  Given the lemma, we have the following intervals of vanishing and non-vanishing homology for $VR(n; r)$.

\begin{theorem}[Intervals of vanishing and non-vanishing, random Vietoris-Rips complex] Fix $k \ge 1$.  For a random Vietoris-
Rips complex on a uniform distribution on a smoothly bounded convex body in $\R^d$,
\begin{enumerate}
\item if $$r = o \left(n^{-\frac{2k+2}{d(2k+1) }}\right) \mbox{ or } r = \omega \left( ( \log n /n )^{1/d} \right)$$ then a.a.s.\ 
$H_k =0$, and
\item if $$r = \omega \left(n^{-\frac{2k+2}{d(2k+1) }} \right) \mbox{ and } r = o \left( ( \log n /n )^{1/d} \right)$$ then a.a.s.\ 
$H_k \neq 0$.
\end{enumerate}
\end{theorem}

Similarly for $C(n,r)$ , we have the following.

\begin{theorem}[Intervals of vanishing and non-vanishing, random \v{Cech} complex] Fix $k \ge 1$.  For a random \v{C}ech 
complex on a uniform distribution on a smoothly bounded convex body,
\begin{enumerate}
\item if $$r = o \left(n^{-\frac{k+2}{d(k+1) }}\right) \mbox{ or } r = \omega \left( ( \log n /n )^{1/d} \right)$$ then a.a.s.\ $H_k 
=0$, and
\item if $$r = \omega \left(n^{-\frac{k+2}{d(k+1) }} \right) \mbox{ and } r = o \left( ( \log n /n )^{1/d} \right)$$ then a.a.s.\ 
$H_k \neq 0$.
\end{enumerate}
\end{theorem}

\begin{proof}  In both cases, (1) follows from the results we have established in the sparse regime.  The point of the Lemma is that as 
long as $r$ falls in this intermediate regime, there is a.a.s.\ at least one connected component homeomorphic to the sphere $S^k$, 
hence contributing to homology $H_k$.
\end{proof}

\section{Further directions} \label{sect:further}

%

From the point of view of applications to topological data analysis, the thing that is most needed is results for statistical persistent homology \cite
{Carlsson}. Bubenik and Kim computed persistent homology for i.i.d.\ uniform random points in the interval \cite{Bubenik} applying the 
theory of order statistics, but so far these are some of the only detailed results for persistent homology of randomly sampled points.
(More recently Bubenik, Carlsson, Kim, and Luo discussed recovering persistent homology of a manifold with respect to some fixed function by data smoothing with kernels, and then applying stability for persistent homology \cite{Bubenik2}.)

The theorems in this article have implications for statistical persistent homology. In 
particular, we have bounded the number of nontrivial homology classes, and since almost all
of the homology comes from vertex minimal spheres, almost all classes should not persist for long.
What one might like is to rule out homology classes that persist for a long time altogether.
Such a theorem would be an important step toward quantifying the statistical significance of persistent homology.

All the results here are stated for Euclidean space, but we believe this is mostly a  matter of convenience.  Analogous results for homology should hold for $d$-dimensional compact Riemannian manifolds.  The manifold will contribute its own homology in the supercritical regime, but for most functions $r=r(n)$ this will be overwhelmed by noise, since $E[\beta_k] \to \infty$ and the homology of the manifold itself will be finite dimensional.  In contrast, one would expect persistent homology to detect the homology of the manifold itself.

Although we have bounded Betti numbers here, coefficients have not come into play.  It seems more refined tools are needed to detect the torsion in $\Z$-homology of random complexes. (This comes up for other kinds of random simplicial complexes as well; see for example \cite{BHK}.)

Finally, we comment that the topological properties studied here are not monotone, the results 
suggest strongly that they are roughly unimodal. But can this be made more precise? For example, can one show that for sufficiently 
large $n$, $E[\beta_k]$ is actually a monotone function of $r$? Similar statistically unimodal behavior in random homology has been 
previously observed in \cite{Kahle_neighborhood} and \cite{Kahle_clique}.

\section*{Acknowledgements}

I gratefully acknowledge Gunnar Carlsson and Persi Diaconis for their mentorship and support, and for suggesting this line 
of inquiry.

I would especially like to thank an anonymous referee for a careful reading of an earlier version of this article and for several 
suggestions which significantly improved it.  I also thank Yuliy Barishnikov, Peter Bubenik, Mathew Penrose, and Shmuel Weinberger for helpful conversations, and Afra Zomorodian for computing and plotting homology of a random geometric complex.

This work was supported in part by Stanford's NSF-RTG grant. Some of this work was completed at the 
workshop in Computational Topology at Oberwolfach in July 2008.

\section*{Appendix: discrete Morse theory}

In this section we briefly introduce terminology of discrete Morse theory and state the main theorem.  For a more complete 
introduction to the subject we refer the reader to \cite{Forman}.
 
For two faces $\sigma, \tau$ of a simplicial complex, we write $\sigma \prec \tau$ if $\sigma$ is a face of $\tau$ of codimension 1.

\begin{definition} A discrete vector field $V$ of a simplicial complex $\Delta$ is a collection of pairs of faces of $\Delta$
$\{ \alpha \prec \beta \}$ such that each face is in at most one pair.
\end{definition}

Given a discrete vector field $V$, a {\it closed $V$-path} is a sequence of faces

$$\alpha_0 \prec \beta_0 \succ \alpha_1 \prec \beta_1 \succ \ldots \prec \beta_{n} \succ \alpha_{n+1}, $$
with $\alpha_{i+1} \neq \alpha_{i}$ such that $\{ \alpha_i \prec \beta_i \} \in V$ for $i=0, \ldots, n$ and $\alpha_{n+1}=\alpha_o$. (Note that $\{ \beta_i \succ 
\alpha_{i+1} \} \notin V$ since each face is in at most one pair.) We say that $V$ is a {\it discrete gradient vector field} if there are no 
closed $V$-paths.

Call any simplex not in any pair in $V$ {\it critical}. Then the main theorem is the following \cite{Forman}.

\begin{theorem}[Fundamental theorem of discrete Morse theory] \label{morse} Suppose $\Delta$ is a simplicial complex with a 
discrete gradient vector field $V$. Then $\Delta$ is homotopy equivalent to a CW complex with one cell of dimension $k$ for each 
critical $k$-dimensional simplex.
\end{theorem} 

Simply counting cells is an extremely coarse measure of the topology a complex, but it can be enough to completely determine homotopy type; for example a CW complex with one $0$-cell and all the rest of its cells $d$-dimensional is a wedge of $d$-spheres.

In all cases, if $f_k$ is the number of cells of dimension $k$, then the definition of cellular homology gives that $\beta_k \le f_k$, and this is the main fact that we exploit in Sections \ref{sect:sup} and \ref{sect:con} to bound the expected dimension of homology.

\bibliographystyle{plain}

\end{document}